\documentclass[11pt]{article}
\usepackage{graphicx}
\usepackage{epsfig}
\usepackage{latexsym}
\usepackage{amsmath, amsthm, amssymb, mathrsfs}

\newcommand{\beq}{\begin{equation}}
\newcommand{\eeq}{\end{equation}}

\newtheorem{theorem}{Theorem}
\newtheorem{remark}{Remark}
\newtheorem{lemma}{Lemma}
\newtheorem{corollary}{Corollary}
\newtheorem{proposition}{Proposition}
\newtheorem{definition}{Definition}

\def\ba{\begin{array}}
\def\ea{\end{array}}

\newcommand{\conp}{\stackrel{p}{\rightarrow}}

\begin{document}
\pagestyle{myheadings} \markboth{Diversity and its Estimation}{Diversity and its Estimation}\markright{Diversity and its Estimation}

\title{Entropic Representation and Estimation of Diversity Indices}
\author{ Zhiyi Zhang\footnote{E-mail address: zzhang@uncc.edu} and Michael Grabchak\footnote{E-mail address: mgrabcha@uncc.edu}
\\
Department of Mathematics and Statistics \\ University of North Carolina at Charlotte\\
Charlotte, NC 28223 }
\date{}

\maketitle
\begin{abstract}
This paper serves a twofold purpose. First, a unified perspective on diversity indices is introduced based on an entropic basis. It is shown that the class of all linear combinations of the entropic basis, referred to as the class of linear diversity indices, covers a wide range of diversity indices used in the literature. Second, a class of estimators for linear diversity indices is proposed and it is shown that these estimators have rapidly decaying biases and asymptotic normality.\\

\noindent\textbf{keywords:} diversity indices; entropy; entropic basis; nonparametric estimation; asymptotic normality.
\end{abstract}

\section{Introduction.}

Diversity is an important notion in a variety of scientific disciplines.  Historically, the interest stems from ecological applications, where the diversity of species in an ecosystem is a relevant issue.  Other applications include cancer research, where the interest is in the diversity of types of cancer cells in a tumor, and linguistics,  where it is in the diversity of an author's vocabulary.  More generally, in information science, one is interested in the diversity of letters drawn from some alphabet. While there is little controversy on the literal meaning of the word ``diversity'', the best way to quantify this concept is a matter of some dispute. For this reason many indices of diversity have been proposed in the literature.  Perhaps the earliest are Shannon's entropy introduced in Shannon (1948) and Simpson's index introduced in Simpson (1949). Many other popular diversity indices have been proposed, including Emlen's index, the Gini-Simpson index, Hill's diversity number, R\'{e}nyi's entropy, and 
Tsallis entropy. These are, respectively, named after the authors of Emlen (1973), Gini (1912), Hill (1973), R\'{e}nyi (1961), and Tsallis (1988). Comprehensive discussions are offered by many, see {\it e.g.}, Magurran (1988) and Marcon (2013). It is quite clear that, when it comes to the question of what constitutes a mathematical diversity index, the ideas are diverse. This paper offers a unified perspective for all of the above mentioned diversity indices based on a re-parameterization and establishes a 
general nonparametric estimation procedure for these indices.

In the following section we give a general discussion of diversity indices.  In Section \ref{sec: Entropic Basis} we show that all diversity indices (satisfying very general regularity conditions) are transforms of a relatively simple class of diversity indices, which we call the entropic basis. Further, we show that most diversity indices used in the literature are either linear combinations (possibly infinite) of the entropic basis or equivalent to such indices.  In Section \ref{sec: estimation} we give a general non-parametric approach for estimating such linear diversity indices. We derive an estimator, which, under general conditions, is consistent, asymptotically normal, and has a bias that decays exponentially fast. In Section \ref{sec: estimate Renyi} we show how to extend these results to R\'enyi's entropy, which is not a linear diversity index, but only equivalent to one.

\section{Preliminaries}

Let the species in a population be denoted by letters of a countable alphabet $\mathscr{X}=\{\ell_{k};k\geq 1\}$ and let the abundances of the species be represented by a probability distribution $\mathbf{p}=\{p_{k};k\geq1\}$ associated with these letters. For simplicity of notation, and when there is no risk of ambiguity, let $K=\sum_{k\geq1}1[p_{k}>0]$, where $1[\cdot]$ is the indicator function, be referred to as the cardinality of $\mathscr{X}$, which may be finite or countably infinite.

An index of diversity $\theta=\theta(\mathscr X,\mathbf p)$ is a function taking the alphabet $\mathscr X$ and the probability distribution $\mathbf p$ into the extended real line $[-\infty,\infty]$. What additional conditions $\theta$ must satisfy is a matter of opinion or need depending on the particular population features of interest. The discussion of this paper starts with the following set of axioms:
\beq
    \begin{array}{rcl}\nonumber
    A_{01}:& & \mbox{$\theta$ is fully determined by $\mathbf{p}=\{p_{k};k\ge1\}$.} \\
    A_{02}:& & \mbox{$\theta$ is invariant under any permutation on the index set $\{k;k\geq 1\}$.}
    \end{array}
\eeq
We will refer to these axioms as $\mathscr{A}_{0}=\{A_{01},A_{02}\}$. Note that Axiom $A_{01}$ implies that $\theta=\theta(\mathscr X,\mathbf p)=\theta(\mathbf p)$ does not depend on $\mathscr X$. The axiomatic set $\mathscr{A}_{0}$ is by no means sufficient for a meaningful diversity index, but it provides a minimal constraint for further discussion. Most, if not all, popular diversity indices studied in the literature satisfy $\mathscr{A}_{0}$. These include: 
\begin{enumerate}
 \item Simpson's index $\lambda=\sum_{k\geq 1}p_{k}^{2}$,
 \item Gini-Simpson index $1-\lambda=\sum_{k\geq 1}p_{k}(1-p_{k})$,
 \item Shannon's entropy $H=-\sum_{k\geq 1}p_{k}\ln(p_{k})$,
 \item R\'{e}nyi's entropy $H_{r}=(1-r)^{-1}\ln\left(\sum_{k\geq 1}p_{k}^{r}\right)$ for any $r>0, r\ne1$,
 \item Tsallis' entropy $T_{r}=(1-r)^{-1}\left(\sum_{k\geq 1}p_{k}^{r}-1\right)$ for any $r>0, r\ne1$
 \item Hill's diversity number $N_{r}=\left(\sum_{k\geq 1}p_{k}^{r}\right)^{1/(1-r)}$ for any $r>0, r\ne1$,
 \item Emlen's index $D=\sum_{k\ge1}p_ke^{-p_k}$, and
 \item the richness index $K=\sum_{k\geq 1}1[p_{k}>0]$.
\end{enumerate}

Given the abundance of definitions of various diversity indices, let us define a notion of equivalence between two diversity indices. 

\begin{definition}\label{defn: equiv} Two diversity indices, $\theta_{1}$ and $\theta_{2}$, are said to be equivalent if and only if there exists a strictly increasing function $g$ such that $\theta_{1}=g(\theta_{2})$. The equivalence is denoted by $\theta_{1}\Leftrightarrow \theta_{2}$.
\end{definition}

Noting that if $g$ is strictly increasing then so is its inverse $g^{-1}$, it is clear that the above definition is symmetric with respect to $\theta_{1}$ and $\theta_{2}$. Two diversity indices are equivalent if they agree on whether one population is more diverse than another, but they may not agree on what the actual difference is.  With this rather trivial notion of equivalence much of the superficial redundance among various indices in the literature can be erased. For example, R\'{e}nyi's entropy $H_r$, Tsallis entropy $T_r$, and Hill's diversity number $N_r$ are equivalent to each other. Further, for $r\in(0,1)$ they are equivalent to a core index 
\beq h_r = \sum_{k\geq 1}p_{k}^{r}, \ \ r>0, r\ne1,
\label{renyieqiv}
\eeq
which we will refer to as {\it R\'{e}nyi's equivalent entropy}.  Although these indices are not equivalent to $h_r$ for $r>1$, they are nevertheless continuous transformations of $h_r$. This fact will be useful for statistical estimation, see Section \ref{sec: estimate Renyi} below.

\section{Entropic Basis}\label{sec: Entropic Basis}

In this section we give our first main result.  We introduce a class of diversity indices and show that they are the building blocks from which all diversity indices satisfying the axioms of $\mathscr A_0$ are made. For a given $\mathbf{p}=\{p_{k};k\ge1\}$ and an integer pair $(u\geq 1, v\geq 0)$, let $\zeta_{u,v}=\sum_{k\geq 1}p_{k}^{u}(1-p_{k})^{v}$. We will refer to
\beq
       \boldsymbol\zeta=\{\zeta_{u,v};u\geq 1, v\geq 0\}
\label{zetauv}
\eeq  as the family of the {\it generalized Simpson's diversity indices} after Zhang and Zhou (2010). Furthermore we will refer the sub-family
\beq
\boldsymbol\zeta_{1}=\left\{\zeta_{1,v};v\geq 0\right\}
\label{zeta1v}
\eeq as the {\it entropic basis} with regard to $\mathscr{X}$ and $\mathbf{p}$.

\begin{remark}According to Zhang and Zhou (2010), every member of $\boldsymbol\zeta$ can be expressed as
$\zeta_{u,v}=\sum_{k\geq 1}\left[p_{k}^{u-1}\right]\left[p_{k}(1-p_{k})\right]\left[(1-p_{k})^{v-1}\right]$, which may be regarded as a weighted version of the Gini-Simpson diversity index, with weights $\left[p_{k}^{u-1}\right]\left[(1-p_{k})^{v-1}\right]$ for various choices of $u$ and $v$ at the user's discretion, and hence the term  {\it generalized Simpson's diversity indices}.
\end{remark}

\begin{remark}
Zhang (2012) established an alternative representation of Shannon's entropy $H=\sum_{v=1}^{\infty}v^{-1}\zeta_{1,v}$ provided that $H<\infty$. This is a linear form of $\boldsymbol\zeta_{1}$, hence the term {\it entropic basis}.
\end{remark}

We now show that any diversity index satisfying the axioms of $\mathscr{A}_{0}$ must be a function of $\zeta_{1,v}$, for $v=0,1,\dots$, {\it i.e.} of the members of the entropic basis.

\begin{theorem} Given an alphabet $\mathscr{X}=\{\ell_{k}; k\geq 1\}$ and an associated probability distribution $\mathbf{p}=\{p_{k}; k\geq 1\}$, a diversity index $\theta$ satisfying $\mathscr{A}_{0}$ is fully 
determined by the entropic basis $\boldsymbol\zeta_{1}=\left\{\zeta_{1,v};v\geq 0\right\}$.
\label{th1}
\end{theorem}

A proof of Theorem \ref{th1} requires the following lemma due to Zhang and Zhou (2010).

\begin{lemma} The distribution $\mathbf{p} = \{p_k;k\ge1\}$ on $\mathscr{X}$ and the family of generalized Simpson's diversity indices $\boldsymbol\zeta=\left\{\zeta_{u,v};u\geq 1, v\geq 0\right\}$ uniquely determine each other up to a permutation of the index set $\{k;k\geq 1\}$.
\label{lemma1}
\end{lemma}

\begin{proof}[Proof of Theorem \ref{th1}.] By Lemma \ref{lemma1}, $\boldsymbol\zeta$ determines $\mathbf{p}$ up to a permutation, and this fully determines a diversity index $\theta$ satisfying $\mathscr A_{0}$.  
It remains to show that every element of $\boldsymbol\zeta$ is fully determined by $\boldsymbol\zeta_{1}$. Toward that end, we note that for any pair of fixed integers $u\geq 2$ and $v\geq0$
\begin{eqnarray*}
    \zeta_{u,v} &=& \sum_{k\geq 1}p_{k}[1-(1-p_{k})]^{u-1}(1-p_{k})^{v} \\
       &=& \sum_{k\geq 1}p_{k}\left[\sum_{i=0}^{u-1}(-1)^{i}\binom{u-1}{i}(1-p_{k})^{i}\right](1-p_{k})^{v} \\
           &=& \sum_{i=0}^{u-1}(-1)^{i}\binom{u-1}{i}\left[\sum_{k\geq 1}p_{k}(1-p_{k})^{v+i}\right] \\
       &=& \sum_{i=0}^{u-1}(-1)^{i}\binom{u-1}{i}\zeta_{1,v+i},
\end{eqnarray*}
which completes the proof.
\end{proof}

The statement of Theorem \ref{th1} holds true for any probability distribution $\mathbf{p}$ regardless of whether $K$ is finite or infinite. Theorem \ref{th1} essentially offers a re-parameterization of $\mathbf p$ (up to a permutation) in terms of $\boldsymbol\zeta_{1}$. This re-parameterization is not just an arbitrary one, it has several statistical implications. First of all, every element of $\boldsymbol\zeta_1$ contains information about the entire distribution and not just one frequency $p_k$. This helps to deal with the problem of estimating probabilities of unobserved species. Second, for a random sample of size $n$, there are very good estimators of $\zeta_{1,v}$ for $v=0,1,2,\dots,n-1$. These are given in Zhang and Zhou (2010) and are discussed below.

While, in general, a diversity index can be any transformation of the entropic basis, in practice, most commonly used indices correspond to transformations of a fairly simple form. Most diversity indices either belong to, or are equivalent to ones that belong to the following class.

\begin{definition}\label{defn: linear div index}
 A diversity index $\theta$ is said to be a linear diversity index if it is a linear combination of the elements of the entropic basis, {\it i.e.} 
\begin{eqnarray}\label{eq: lin div index}
\theta = \theta(\mathbf p) = \sum_{v=0}^\infty w_v \zeta_{1,v} = \sum_{v=0}^\infty w_v \sum_{k\ge1} p_k(1-p_k)^v
\end{eqnarray}
for any choice of weights $w_v$ such that, for every $\mathbf p$, the sum either converges or diverges to $\pm\infty$.
\end{definition}
Definition \ref{defn: linear div index} essentially encircles a sub-class of indices among all functions of $\boldsymbol\zeta_{1}$, {\it i.e.}, all diversity indices satisfying $\mathscr{A}_{0}$.
While there are no fundamental reasons why a search of a good diversity index should be restricted to this sub-class, it happens to cover all of the popular indices that we have come across in the literature, up to the equivalence relationship given in Definition \ref{defn: equiv}. These include:
\[ \begin{array}{rrclcl}
    \mbox{Simpson's index:} & \lambda &=& \sum_{k\geq 1}p_{k}^{2}&=&\zeta_{1,0}-\zeta_{1,1} \\ &&& \\
    \mbox{Gini-Simpson index:} & 1-\lambda &=&\sum_{k\geq 1}p_{k}(1-p_{k})&=& \zeta_{1,1} \\ &&& \\
    \mbox{Shannon's entropy:} &  H&=& -\sum_{k\geq 1}p_{k}\ln(p_{k})&=&\sum_{v= 1}^{\infty}\frac{1}{v}\hspace{0.2em}\zeta_{1,v} \\  &&& \\
    \mbox{R\'{e}nyi equiv.\ entropy:} &  h_{r}&=& \sum_{k\geq 1}p_{k}^{r}&=&\zeta_{1,0}+\sum_{v=1}^{\infty}\prod_{i=1}^{v}\left(\frac{i-r}{i}\right)\hspace{0.2em}\zeta_{1,v}, \\&&&\\
    \mbox{Emlen's index:} & D &=&\sum_{k\geq 1}p_ke^{-p_k}&=& \sum_{v=0}^\infty \frac{e^{-1}}{v!}\zeta_{1,v}\\&&&\\
 \mbox{Richness index:} & K&=&\sum_{k\geq 1}1[p_{k}>0] &=&  \sum_{v=0}^\infty \zeta_{1,v}\\&&&\\
    \mbox{Gen.\ 
    Simpson's index:} &  \zeta_{u,\hspace{0.2em}m} &=& \sum_{k\geq 1}p_k^u(1-p_k)^m&=&\sum_{v=0}^{u-1}(-1)^{v}\binom{u-1}{v}\hspace{0.2em}\zeta_{1,\hspace{0.2em}m+v}.
\end{array}
\] 
Note that $\zeta_{1,0}=\sum_{k\ge1}p_k(1-p_k)^0\equiv1$. We also note that Tsallis' entropy is a linear diversity index.  The form of its weights are very similar to those of R\'enyi's equivalent entropy. All of the representations above can be verified using Taylor expansions.  For example, for the richness index, which is the total number of species in a population, we have
\begin{eqnarray*}
    K&=& \sum_{k\geq 1}1[p_{k}>0]\frac{p_{k}}{1-(1-p_{k})}\\
    &=&\sum_{k\geq 1}1[p_{k}>0]p_{k}\sum_{v=0}^{\infty}(1-p_{k})^{v} \\
    &=& \sum_{v=0}^{\infty}\sum_{k\geq 1}1[p_{k}>0]p_{k}(1-p_{k})^{v}=\sum_{v=0}^{\infty}\zeta_{1,v}.
 \end{eqnarray*}

It is not difficult to see that all linear diversity indices discussed above are of the general form
\begin{eqnarray}\label{eq: alt lin div index}
\theta = \sum_{k\ge1} p_k h(p_k),
\end{eqnarray}
where $h$ has a Taylor expansion around $1$ with radius of convergence at least $1$. This is not a coincidence. If $\theta$ satisfies \eqref{eq: lin div index} and $h(t) =  \sum_{v=0}^\infty w_v (1-t)^v$ then $\theta$ is necessarily of the form \eqref{eq: alt lin div index}, and, of course, the converse of this statement holds.

\begin{remark}
 In the literature of diversity indices, it is generally thought that the richness indices, {\it e.g.}, $K$, and the evenness indices, {\it e.g.}, Gini-Simpson's $1-\lambda$, are two qualitatively different types of indices, see {\it e.g.}, Peet (1974) and Heip, Herman, and Soetaert (1998). It may be interesting to note that, in the perspective of the entropic basis, they are both linear diversity indices and merely differ in the weighting scheme $h(p)$ in (\ref{eq: alt lin div index}).
\end{remark}

\section{Estimation of Linear Diversity Indices.}\label{sec: estimation}

In this section we discuss nonparametric estimation of linear diversity indices.  Assume that $X_1,X_2,\dots,X_n$ are independent and identically distributed (iid) from $\mathscr X$ according to $\mathbf p$. We want to estimate
\begin{eqnarray}\label{eq: theta}
\theta=\theta(\mathbf p) = \sum_{v=0}^\infty w_v \zeta_{1,v} = \sum_{v=0}^\infty w_v \sum_{k\ge1} p_k(1-p_k)^v.
\end{eqnarray}
We assume that $\{w_v:v\ge0\}$ has been chosen but that $\mathbf p=\{p_k:k\ge1\}$ 
is unknown. We will make the following assumptions:
\begin{enumerate}
\item there is an $M>0$ such that $|w_v|\le M$ for all $v\ge0$, and
\item $K<\infty$.
\end{enumerate}
These conditions guarantee that the sum in \eqref{eq: theta} always converges.  Note that the assumption that $|w_n|\le M$ is satisfied by all of the linear diversity indices discussed in Section \ref{sec: Entropic Basis}.
Note further that we are not assuming that $K$ is known, only that it is known that $K<\infty$.  This is realistic in many applications including ecology, where there is a finite (even if very large) number of species.

For simplicity of notation assume the frequencies $p_k$ are ordered such $p_k>0$ for $k=1,2,\dots,K$ and that $p_k=0$ for $k>K$. For $t\in[0,1]$ let
$$
h(t) = \left\{\begin{array}{ll}
  \sum_{v=0}^\infty w_v  (1-t)^v & \mbox{if }t\in(0,1]\\
0& \mbox{if }t=0
\end{array}\right..
$$
In this case $w_v = (-1)^vh^{(v)}(1)/v!$ and
$$
\theta=\sum_{k=1}^K p_k h(p_k).
$$

Let $\{x_k=\sum_{i=1}^n 1[X_i=\ell_k]\}$ be the sequence of observed counts in our sample and let $\{\hat p_k = x_k/n\}$ be the sample proportions. Perhaps the most intuitive estimator of $\theta$ is the plug-in estimator given by
\begin{eqnarray}\label{eq: plug in}
\hat \theta_n = \sum_{v=0}^\infty w_v \sum_{k=1}^K \hat p_k(1-\hat p_k)^v = \sum_{k=1}^K \hat p_k h(\hat p_k).
\end{eqnarray}
However, it is well known that, in many important situations, the plug-in estimator has a bias that decays very slowly. For instance, in the case of Shannon's entropy ({\it i.e.}, when $w_0=0$ and $w_v=1/v$ for $v\ge1$), the bias decays no faster that $\mathcal O(1/n)$, see {\it e.g.}, Paninski (2003). We now propose another estimator, which has a bias that always decays at least exponentially fast. Our approach is influenced by the estimator of Shannon's entropy derived in Zhang (2012).

First note that
\begin{eqnarray}\label{eq: second term}
\theta = w_0+\sum_{v=1}^{n-1} w_v \sum_{k=1}^K p_k (1-p_k)^v+\sum_{v=n}^\infty w_v \sum_{k=1}^K p_k (1-p_k)^v =: \eta_{n}+B_{2,n}.
\end{eqnarray}
From Zhang and Zhou (2010), we know that an unbiased estimator of $\eta_n$ is given by
\begin{eqnarray}\label{eq: estimator}
\hat \theta_n^\sharp &=& w_0+\sum_{v=1}^{n-1} w_v \sum_{k=1}^K \hat p_k \prod_{j=1}^{v}\left(1-\frac{x_k-1}{n-j}\right)\nonumber\\
&=& w_0+\sum_{k=1}^K \hat p_k \sum_{v=1}^{n-x_k} w_v \prod_{j=1}^{v}\left(1-\frac{x_k-1}{n-j}\right) = w_0+\sum_{k=1}^K \hat \theta^\sharp_{n,k},
\end{eqnarray}
where
\begin{eqnarray}\label{eq: estimator k term}
\hat \theta_{n,k}^\sharp = \hat p_k \sum_{v=1}^{n-x_k} w_v \prod_{j=1}^{v}\left(1-\frac{x_k-1}{n-j}\right)= \hat p_k \sum_{v=1}^{\infty} w_v \prod_{j=1}^{v}\left(1-\frac{x_k-1}{n-j}\right).
\end{eqnarray}
It may, at first, appear that one needs to know $K$ in order to evaluate this estimator. However, if a category $k$ is not observed then $\hat p_k=0$ and hence $\hat\theta_{n,k}^\sharp=0$ and does not need to be included in the sum. Thus one does not need to know the value of $K$ in order to evaluate this estimator. A similar comment holds for the plug-in estimator given in \eqref{eq: plug in}. 

By construction, the bias of the estimator $\hat \theta_n^\sharp$ is given by $B_{2,n}$, and letting $p_\wedge=\min\{p_k:1\le k\le K\}$ we see that
$$
\left|B_{2,n}\right|\le M K(1-p_{\wedge})^n,
$$
which decays exponentially fast in $n$. We note that in the case of Shannon's entropy this estimator corresponds with the estimator introduced in Zhang (2012) and Zhang (2013).  For that estimator, an approach to further reduce the bias was presented in Zhang and Grabchak (2013).  One can modify that approach for our more general situation.  This will be dealt with in a future work.

Next, we will establish that $\hat \theta^\sharp_n$ is a consistent and asymptotically normal estimator of $\theta$. Along the way, we will show the corresponding results for the plug-in estimator $\hat\theta_n$. Our approach is similar to the one used in Zhang (2013) to prove the asymptotic normality of an estimator of Shannon's entropy.

Let us define the $(K-1)$-dimensional vectors
$$
\mathbf v=(p_1,\cdots,p_{K-1})^{\tau} \mbox{ and } \hat {\mathbf v}_n=(\hat{p}_{1},\cdots,\hat{p}_{K-1})^{\tau},
$$
and note that $\hat{\mathbf  v}_n\conp \mathbf  v$ as $n\to\infty$. Moreover, by the multivariate normal approximation to the multinomial distribution
\begin{eqnarray}\label{lemma0}
   \sqrt{n}(\hat {\mathbf  v}_n-\mathbf v)\stackrel{L}{\rightarrow} MVN(0, \Sigma(\mathbf  v)),
\end{eqnarray}
where $\Sigma(\mathbf  v)$ is the $(K-1)\times (K-1)$ covariance matrix given by
\begin{eqnarray}\label{sigmamatrix}
\Sigma(\mathbf  v)=\left(\begin{array}{cccc}
                p_{1}(1-p_{1}) & -p_{1}p_{2}&\cdots & -p_{1}p_{K-1}  \\
                -p_{2}p_{1}& p_{2}(1-p_{2})  &\cdots & -p_{2}p_{K-1} \\
                \cdots&\cdots  &\cdots & \cdots \\
                -p_{K-1}p_{1}&- p_{K-1}p_{2}  &\cdots & p_{K-1}(1-p_{K-1})
                \end{array}\right).
\end{eqnarray}
Let
\begin{eqnarray*}
G(\mathbf v) =\sum_{k=1}^{K-1}p_k h(p_k) + \left(1-\sum_{k=1}^{K-1}p_k\right) h\left(1-\sum_{k=1}^{K-1} p_k\right)
\end{eqnarray*}
and
$$
g(\mathbf v):=\nabla G(\mathbf v) = \left(\frac{\partial}{\partial p_{1}}G(\mathbf v),\cdots,\frac{\partial}{\partial p_{K-1}}G(\mathbf v)\right)^{\tau}.
$$
For each $j$, $j=1,\cdots,K-1$, we have
\begin{eqnarray*}
\frac{\partial}{\partial p_j} G(\mathbf v) &=& h(p_j)+p_jh'(p_j)-h\left(1-\sum_{k=1}^{K-1} p_k\right)-\left(1-\sum_{k=1}^{K-1} p_k\right)h'\left(1-\sum_{k=1}^{K-1} p_k\right).
\end{eqnarray*}
The delta method gives the following result.

\begin{proposition} \label{prop: asymp norm MLE}
If $\hat\theta_n$ is the plug-in estimator given by \eqref{eq: plug in} and  $g^{\tau}(\mathbf v)\Sigma(\mathbf v) g(\mathbf v)>0$ then
\begin{eqnarray*}
\sqrt{n}\left(\hat{\theta}_{n}-\theta\right)\left[ g^{\tau}(\mathbf v)\Sigma(\mathbf v) g(\mathbf v)\right]^{-\frac{1}{2}}\stackrel{L}{\rightarrow} N(0, 1).
\end{eqnarray*}
\end{proposition}

\begin{remark}\label{remark: simple cond}
It is well-known that $\Sigma(\mathbf v)$ is a positive definite matrix, see e.g.\ Tanabe and Sagae (1992).  For this reason, the condition $g^{\tau}(\mathbf v)\Sigma(\mathbf v) g(\mathbf v)>0$ is equivalent to the condition that $g(\mathbf v)\ne0$.  The question of when this holds depends of the function $h$. In the case of entropy (when $h(t)=\log t$) and R\'enyi's equivalent entropy (when $h(t)=t^{r-1}$) it is easy to verify that $g(\mathbf v) = 0$ if and only if $p_k=1/K$ for $k=1,2,\dots,K$.
\end{remark}

In order to use Proposition \ref{prop: asymp norm MLE} in applications we need to be able to estimate $g^{\tau}(\mathbf v)\Sigma(\mathbf v) g(\mathbf v)$. By the continuous mapping theorem, we can estimate $\Sigma(\mathbf v)$ by $\Sigma(\hat {\mathbf v})$. However, $g(\hat {\mathbf v})$ may not be defined when there are species that have not been observed in the sample. To deal with this, for $\mathbf x=(x_1,x_2,\dots,x_{K-1})^\tau\in[0,1]^{K-1}$ with $\sum_{i=1}^{K-1}x_i\in[0,1]$ define
$$
\bar g(\mathbf x) = (\bar g_{1}(\mathbf x), \bar g_2(\mathbf x),\dots,\bar g_{K-1}(\mathbf x))^\tau
$$
where for $j=1,2,\dots,K-1$
$$
\bar g_j(\mathbf x) = \left\{\begin{array}{ll}
   \frac{\partial}{\partial x_j} G(\mathbf x) & \mbox{if } x_j>0\\
0 & \mbox{otherwise}
             \end{array}
\right..
$$
Since $\hat {\mathbf v}_n\conp \mathbf v$, $\bar g$ is continuous for all $\mathbf x\in(0,1]^{K-1}$, and $\mathbf  v \in(0,1]^{K-1}$, the continuous mapping theorem implies that $\bar g(\hat {\mathbf v}_n)$ is a consistent estimator of $g(\mathbf v)$. From this and Slutsky's Theorem we get the following.

\begin{corollary} \label{cor: estim asympt norm MLE}
If $\hat\theta_n$ is the plug-in estimator given by \eqref{eq: plug in} and  $g^{\tau}(\mathbf v)\Sigma(\mathbf v) g(\mathbf v)>0$ then
\begin{eqnarray*}
 \sqrt{n}\left(\hat{\theta}_{n}-\theta\right)  \left[\bar g^{\tau}(\hat {\mathbf v}_n)\Sigma(\hat {\mathbf v}_n)\bar g(\hat {\mathbf v}_n)\right]^{-\frac{1}{2}}
 \stackrel{L}{\rightarrow} N(0, 1).
\end{eqnarray*}
\end{corollary}

Since both $\Sigma(\hat {\mathbf v}_n)$ and $\bar g(\hat {\mathbf v}_n)$ have zeros in locations that correspond to unobserved species, we can pretend that these species do not exist for the purposes of estimating $\bar g^{\tau}(\hat {\mathbf v}_n)\Sigma(\hat {\mathbf v}_n)\bar g(\hat {\mathbf v}_n)$. For this reason, we do not actually need to know $K$ to evaluate this quantity. We now extend our results to the estimator defined in \eqref{eq: estimator}.

\begin{theorem} \label{thetheorem}
If $\hat\theta_n^\sharp$ is the estimator given by \eqref{eq: estimator} and $g^{\tau}(\mathbf v)\Sigma(\mathbf v) g(\mathbf v)>0$ then
 \begin{eqnarray*}
\sqrt{n}\left(\hat{\theta}_{n}^{\sharp}-\theta\right)\left[g^{\tau}(\mathbf v)\Sigma(\mathbf v) g(\mathbf v)\right]^{-\frac{1}{2}}\stackrel{L}{\rightarrow} N(0, 1).
\end{eqnarray*}
\end{theorem}

Before giving the proof, we state the following corollary. Its proof is similar to that of Corollary \ref{cor: estim asympt norm MLE}.

\begin{corollary} \label{cor: estim asympt norm sharp}
If $\hat\theta_n^\sharp$ is the estimator given by \eqref{eq: estimator} and $g^{\tau}(\mathbf v)\Sigma(\mathbf v) g(\mathbf v)>0$ then
\begin{eqnarray*}
 \sqrt{n}\left(\hat{\theta}_{n}^\sharp-\theta\right)  \left[ \bar g^{\tau}(\hat {\mathbf v}_n)\Sigma(\hat {\mathbf v}_n)\bar g(\hat {\mathbf v}_n)\right]^{-\frac{1}{2}}
 \stackrel{L}{\rightarrow} N(0, 1).
\end{eqnarray*}
\end{corollary}

As before, note that we do not need to know $K$ to evaluate $\bar g^{\tau}(\hat {\mathbf v}_n)\Sigma(\hat {\mathbf v}_n)\bar g(\hat {\mathbf v}_n)$. The proof of Theorem \ref{thetheorem} will be based on the following.

\begin{lemma}\label{lemma: rates}
For $p\in[0,1]$ and $n\in\mathbb N$ let
$$
g_n(p) = p \sum_{v=1}^{\lfloor n(1-p)+1\rfloor} w_v \prod_{j=1}^{v}\left(1-\frac{np-1}{n-j}\right)
$$
and let
$$
g(p) =  p \sum_{v=1}^\infty w_v (1-p)^v.
$$
1. If $0<c<d<1$ then
$$
\lim_{n\to\infty}\sup_{p\in[c,d]}\sqrt n|g_n(p)-g(p)|= 0.
$$
2. Let $p_n\in [0,1]$ such that $np_n \in\{0,1,2,\dots,n\}$ then
$$
\sqrt n \left|g_n(p_n)-g(p_n)\right| \le  \sqrt n (n+1) \le 2n^{3/2}.
$$
\end{lemma}

\begin{proof}
Note that
\begin{eqnarray*}
\sqrt n|g_n(p)-g(p)| &\le& M\sqrt{n} p \sum_{v=1}^{\lfloor n(1-p)+1\rfloor}  \left|\prod_{j=1}^{v}\left(1-\frac{np-1}{n-j}\right)-(1-p)^v\right|\\
&& + M\sqrt{n} p \sum_{v=\lfloor n(1-p)+2\rfloor}^{\infty}  (1-p)^v =: M(\Delta_1+\Delta_2).
\end{eqnarray*}
We begin by showing Part 1. Throughout the proof of this part, we assume that $n>2/c$; this ensures that $n(1-p)+1<n-1$ for all $p\in[c,d]$. Fix $v\in\mathbb N$ such that $v\le n(1-p)+1$. Note that
$$
\prod_{j=1}^v\left(1-\frac{np-1}{n-j}\right) = \prod_{j=0}^{v-1}\left(\frac{1-p-\frac{j}{n}}{1-\frac{j+1}{n}}\right)
$$
and thus
\begin{eqnarray*}
\left|\prod_{j=1}^{v}\left(1-\frac{np-1}{n-j}\right)- (1-p)^v\right| = (1-p)^v\left|\prod_{j=0}^{v-1}\left(\frac{1-\frac{j}{n(1-p)}}{1-\frac{j+1}{n}}\right)- 1\right| \le (1-p)^{v-1}\frac{v^2}{n-v},
\end{eqnarray*}
where the inequality follows by the proof of Part 1 of Lemma 2 in Zhang (2013).  Let $V_n = \lfloor n^{1/8}\rfloor$. For large enough $n$, $V_n<\lfloor n(1-d)+1\rfloor$. For such $n$ we have
\begin{eqnarray*}
\Delta_1 &\le& \sqrt{n} p \sum_{v=1}^{\lfloor n(1-p)+1\rfloor}(1-p)^{v-1}\frac{v^2}{n-v}\\
&=& \sqrt{n} p \sum_{v=1}^{V_n}(1-p)^{v-1}\frac{v^2}{n-v}+\sqrt{n} p \sum_{v=V_n+1}^{\lfloor n(1-p)+1\rfloor}(1-p)^{v-1}\frac{v^2}{n-v}=:\Delta_{11}+\Delta_{12}.
\end{eqnarray*}
We have
\begin{eqnarray*}
\Delta_{11} \le \sqrt{n} p \frac{V_n^2}{n-V_n}\sum_{v=1}^\infty (1-p)^{v-1} =  \sqrt{n} \frac{V_n^2}{n-V_n} \le \frac{n^{3/4}}{n-n^{1/8}} \to 0,
\end{eqnarray*}
\begin{eqnarray*}
\Delta_{12} &\le& \sqrt{n} p (1-p)^{V_n}\frac{( n(1-p)+1)^2}{n-( n(1-p)+1)} \sum_{v=1}^{\infty} (1-p)^{v-1}\\
&=& \sqrt{n} (1-p)^{V_n}\frac{[n(1-p)+1]^2}{np-1} \le \sqrt{n} (1-c)^{\lfloor n^{1/8}\rfloor}\frac{[n(1-c)+1]^2}{nc-1}  \to0,
\end{eqnarray*}
and
\begin{eqnarray*}
\Delta_2 \le \sqrt{n} p (1-p)^{\lfloor n(1-p)+2\rfloor} \sum_{v=0}^{\infty}  (1-p)^v \le \sqrt{n} (1-c)^{\lfloor n(1-d)+2\rfloor}\to0.
\end{eqnarray*}

Now to show Part 2.  Note that $p_n\ne0$ implies $p_n\ge1/n$, which means that $\frac{np-1}{n-j}\in[0,1]$ when $j\le n(1-p)+1$. Thus, either $p_n=0$ or  $\prod_{j=1}^{v}\left(1-\frac{np_n-1}{n-j}\right)\in[0, 1]$ when $v\le n(1-p)+1$. This implies that
\begin{eqnarray*}
\Delta_1 &=&  \sqrt{n} p_n \sum_{v=1}^{\lfloor n(1-p_n)+1\rfloor}  \left|\prod_{j=1}^{v}\left(1-\frac{np_n-1}{n-j}\right)-(1-p_n)^v\right| 1_{p_n\ne0} \le \sqrt n (n+1).
\end{eqnarray*}
and
\begin{eqnarray*}
\Delta_2 \le \sqrt{n} p_n \sum_{v=0}^{\infty}  (1-p_n)^v 1_{p_n\ne0}\le \sqrt n.
\end{eqnarray*}
This completes the proof.
\end{proof}

\noindent {\it Proof of Theorem \ref{thetheorem}.} Since $\sqrt{n}\left(\hat{\theta}_{n}^{\sharp}-\theta\right)= \sqrt{n}\left(\hat{\theta}_{n}^{\sharp}-\hat{\theta}_{n}\right)+\sqrt{n}\left(\hat{\theta}_{n}-\theta\right)$, by Proposition \ref{prop: asymp norm MLE} and Slutsky's Theorem it suffices to show that $\sqrt{n}\left(\hat{\theta}_{n}^{\sharp}-\hat{\theta}_{n}\right)\stackrel{p}{\rightarrow}0$.  Further, by Slutsky's Theorem it suffices to show that $\sqrt{n}\left(\hat{\theta}_{n,k}^{\sharp}-\hat{\theta}_{n,k}\right) \stackrel{p}{\rightarrow}0$, where $\hat{\theta}_{n,k}^{\sharp}$ is given by \eqref{eq: estimator k term} and
$$
\hat\theta_{n,k} = \hat p_k \sum_{v=1}^\infty w_v (1-\hat p_n)^v.
$$

We can write
\begin{eqnarray*}
\sqrt{n}\left(\hat{\theta}_{n,k}^{\sharp}-\hat{\theta}_{n,k}\right) &=& \sqrt{n}\left(\hat{\theta}_{n,k}^{\sharp}-\hat{\theta}_{n,k}\right)1_{\hat p_k\le p_k/2} + \sqrt{n}\left(\hat{\theta}_{n,k}^{\sharp}-\hat{\theta}_{n,k}\right)1_{\hat p_k \ge (1+p_k)/2}\\
&& + \sqrt{n}\left(\hat{\theta}_{n,k}^{\sharp}-\hat{\theta}_{n,k}\right)1_{p_k/2 < \hat p_k <  (p_k+1)/2} =: \mathcal A_1+\mathcal A_2.
\end{eqnarray*}
By Part 2 of Lemma \ref{lemma: rates} it follows that
\begin{eqnarray*}
\mathrm E\left|\mathcal A_1\right| &\le& 2n^{3/2} \left[\mathrm P\left(\hat p_k > (1+p_k)/2\right)+ \mathrm P\left(\hat p_k\le p_k/2\right)\right] \\
&\le& 2n^{3/2} \left[\mathrm P\left(|\hat p_k-p_k| \ge (1-p_k)/2\right)+ \mathrm P\left(|\hat p_k-p_k|\ge p_k/2\right)\right]\\
&\le& 4n^{3/2} \left[e^{-n(1-p_k)^2/2}+e^{-np_k^2/2}\right]\to0,
\end{eqnarray*}
where the third line follows by Hoeffding's inequality, see Hoeffding (1963). Thus it follows that $\mathcal A_1\conp0$. By Part 1 of Lemma \ref{lemma: rates} it follows that $\mathcal A_2\conp 0$. \qed

\section{Estimation of R\'enyi's Entropy}\label{sec: estimate Renyi}

The only diversity indices that we have discussed that do not belong to the class of linear diversity indices are R\'enyi's entropy and Hill's diversity number.  However,  both are transformation of R\'enyi's equivalent entropy, $h_{r}$. In this section we will extend Theorem \ref{thetheorem} to R\'enyi's entropy.  We can use a similar approach to extend it to Hill's diversity number.

Fix $r>0$ such that $r\ne1$ and let $\psi_r(t) =(1-r)^{-1}\ln t$. Note that
$$
H_r = \frac{\ln h_r}{1-r} = \psi_r(h_r),
$$
where $H_r$ is R\'enyi's entropy and $h_r$ is R\'enyi equivalent entropy. Let
$$
\hat h_r^\sharp(n) = 1+\sum_{k=1}^K \hat p_k \sum_{v=1}^{n-x_k} w_v \prod_{j=1}^{v}\left(1-\frac{x_k-1}{n-j}\right),
$$
where for $v\ge1$
$$
w_v = \prod_{i=1}^{v}\left(\frac{i-r}{i}\right).
$$
This is the estimator of $h_r$ given by \eqref{eq: estimator}. Let
$$
\hat H_r^\sharp(n) = \frac{\ln \hat h^\sharp_r(n)}{1-r}=\psi_r\left(\hat h^\sharp_r(n)\right).
$$
Since $\psi'_r(t) =t^{-1}(1-r)^{-1}$ the delta method together with Theorem \ref{thetheorem}, Remark \ref{remark: simple cond}, and the fact that $\psi'_r(h_r)>0$ implies the following.

\begin{theorem} \label{thetheorem Renyi}
Provided that there exists a $k\in\{1,2,\dots,K\}$ with $p_k\ne1/K$ 
 \begin{eqnarray*}
\sqrt{n}\left(\hat{H}_{r}^{\sharp}(n)-H_r\right)h_r(1-r)\left[g^{\tau}(\mathbf v)\Sigma(\mathbf v) g(\mathbf v)\right]^{-\frac{1}{2}}\stackrel{L}{\rightarrow} N(0, 1).
\end{eqnarray*}
\end{theorem}

We note that in the case when $r=2$ this result is given in Leonenko and Seleznjev (2010). By arguments similar to the proof of Corollary \ref{cor: estim asympt norm MLE} we get the following.

\begin{corollary} \label{cor: estim asympt norm sharp Renyi}
Provided that there exists a $k\in\{1,2,\dots,K\}$ with $p_k\ne1/K$ 
\begin{eqnarray*}
 \sqrt{n}\left(\hat{H}_{r}^\sharp(n)-\theta\right)  \hat h^\sharp_r(n) (1-r)\left[ \bar g^{\tau}(\hat {\mathbf v}_n)\Sigma(\hat {\mathbf v}_n) \bar g(\hat {\mathbf v}_n)\right]^{-\frac{1}{2}}
 \stackrel{L}{\rightarrow} N(0, 1).
\end{eqnarray*}
\end{corollary}

\section*{acknowledgements}

The authors wish to thank Dr.\ Eric Marcon for correcting a mistake in the formula for $\hat h_r^\sharp(n)$.

\end{document}